\documentclass[10pt, reqno]{amsart}
\usepackage{lipsum}
\usepackage{a4wide}
\usepackage{amssymb} 
\usepackage{amsmath}
\usepackage{amsthm} 
\usepackage{tikz} 
\usepackage{amscd}
\usepackage{hyperref}
\usepackage{enumitem}
\hypersetup{colorlinks,linkcolor={blue},citecolor={blue},urlcolor={red}}
\theoremstyle{plain}
\numberwithin{equation}{section}

\newcommand{\res}{\operatorname{res}}
\newcommand{\ho}{\operatorname{Hom}}

\newtheorem{theorem}{Theorem}[section]
 
\newtheorem{lemma}[theorem]{Lemma}
\newtheorem{remark}[theorem]{Remark}
\newtheorem{proposition}[theorem]{Proposition}

\newtheorem{definition}[theorem]{Definition}

\newtheorem{hypothesis}{Hypothesis}[section]

\title{A note on branching of $V(\rho)$}
\author{ Santosh Nadimpalli and
Santosha Pattanayak}
\date{\today}
\begin{document}
\begin{abstract}
  Let $\mathfrak{g}$ be a complex simple Lie algebra and let $\mathfrak{g}_0$
  be the sub-algebra fixed by a diagram automorphism of $\mathfrak{g}$.
  Let $G$ be the complex, simply-connected, simple algebraic group with
  Lie algebra $\mathfrak{g}$, and let $G_0$ be the connected subgroup
  of $G$ with Lie algebra $\mathfrak{g}_0$.  Let $\rho$ be
  the half sum of positive roots of $\mathfrak{g}$. In this article,
  we give a necessary and sufficient condition for a highest weight
  $\mathfrak{g}_0$-representation $V_0(d\mu)$ to occur in the
  representation $\res_{\mathfrak{g}_0}V(d\rho)$, for any saturation
  factor $d$ of the pair $(G_0, G)$.
  \end{abstract}
  \maketitle
  \section{Introduction}
  Let $\mathfrak{g}$ be a complex simple Lie algebra, and let $\rho$
  be the half sum of positive roots, for some choice of positive roots
  of the root system of $\mathfrak{g}$ with respect to a Cartan
  subalgebra $\mathfrak{h}$ and a Borel subalgebra $\mathfrak{b}$,
  with $\mathfrak{h}\subset \mathfrak{b}$. Let $\mu$ be a dominant
  weight with respect to this chosen set of positive roots, and let
  $V(\mu)$ be the irreducible representation of $\mathfrak{g}$ with
  highest weight $\mu$.  A conjecture of Kostant predicts that the
  representation $V(\mu)$ occurs in $V(\rho)\otimes V(\rho)$ for all
  $\mu\preceq 2\rho$ with respect to the dominance order. In the
  article \cite{components}, the authors show that $V(d\mu)$ occurs in
  $V(d\rho)\otimes V(d\rho)$ for all $\mu\preceq 2\rho$, and for any
  {\it saturation factor} $d$ associated to the Lie algebra
  $\mathfrak{g}$ (see \cite[Definition 11]{kumar_survey}).  The
  results of the article \cite{components} are related to the branching
  of the $\mathfrak{g}\times \mathfrak{g}$ representation
  $V(\rho)\otimes V(\rho)$ to the diagonal embedding of $\mathfrak{g}$
  in $\mathfrak{g}\times \mathfrak{g}$.  In this article, we consider
  a similar branching problem for symmetric pairs associated to
  diagram automorphisms.

  Let $\Phi$ be the set of roots of $\mathfrak{g}$ with respect to
  $\mathfrak{h}$. Let $\{X_\alpha:\alpha\in \Phi\}$ be a Chevalley
  basis for the Lie algebra $\mathfrak{g}$. Let $\theta$ be a
  non-trivial automorphism fixing the pinning
  $(\mathfrak{g}, \mathfrak{b}, \mathfrak{h}, \{X_\alpha\})$. Let
  $\mathfrak{g}_0$ be the fixed point Lie sub-algebra
  $\{X\in \mathfrak{g}:\theta(X)=X\}$. Let $\mathfrak{h}_0$ be the
  Cartan sub-algebra $\mathfrak{g}_0\cap \mathfrak{h}$ of
  $\mathfrak{g}_0$, and let $\mathfrak{b}_0$ be the Borel sub-algebra
  $\mathfrak{g}_0\cap \mathfrak{b}$ of $\mathfrak{g}_0$. Let $\Phi_0$
  be the set of roots of $\mathfrak{g}_0$ with respect to
  $\mathfrak{h}_0$. Let $\rho_0$ be the half sum of positive roots for
  the choice of $\mathfrak{b}_0$ as a Borel sub-algebra of
  $\mathfrak{g}_0$.  Let $G$ be the complex, simply-connected, simple
  algebraic group with Lie algebra $\mathfrak{g}$, and let $G_0$ be
  the connected subgroup of $G$ with Lie algebra $\mathfrak{g}_0$.

  Let $d$ be a saturation factor for the pair $(G_0, G)$ (see
  \cite[Definition 12]{kumar_survey} or see Definition
  \ref{sat_fac}). Our main theorem gives a necessary and sufficient
  condition for a $\mathfrak{g}_0$-representation $V_0(d\mu)$ to occur
  in the restriction $\res_{\mathfrak{g}_0}V(d\rho)$. Let
  $p:\mathfrak{h}^\ast\rightarrow \mathfrak{h}_0^\ast$ be the
  restriction map. A subset $\Gamma$ of $\Phi_0$ is said to be a root
  subsystem if $\Gamma$ is a root system in the Euclidean space
  spanned by $\Gamma$. We do not necessarily assume that $\Gamma$ is
  closed in $\Phi$. For instance, the set of short roots $\Phi_{0,s}$
  of $\Phi_{0}$ is a root subsystem of $\Phi_0$ but it is not
  necessarily closed.
  In this note, we show that $V_0(d\mu)$ occurs in the restriction
  $\res_{\mathfrak{g}_0}V(d\rho)$ if and only if $\mu\preceq p(\rho)$
  in the dominance order of $\Phi_{0,s}$ and $\mu-(2\rho_0-p(\rho))$
  is dominant in $\Phi_{0,s}$. Here, the ordering in $\Phi_{0,s}$ is
  induced by the Borel sub-algebra $\mathfrak{b}_0$ of
  $\mathfrak{g}_0$. The condition that $\mu\preceq p(\rho)$ is
  analogous to $\mu\preceq 2\rho$ in the tensor product case.

  Kostant's conjecture is inspired by the
  $\mathfrak{g}$-representation on the exterior algebra
  $\wedge^{\bullet} \mathfrak{g}$. If $\mathfrak{g}$ has a diagram
  automorphism $\theta$ of order two, we set $\mathfrak{g}_1$ to be
  the $-1$ eigenspace of $\theta$. We then get
  $\wedge^\bullet\mathfrak{g}=\wedge^\bullet\mathfrak{g}_0\otimes
  \wedge^\bullet\mathfrak{g}_1$. Based on the work of Kostant
  (\cite[Proposition 20, Corollary 36]{kostant_clifford}) and
  Panyushev (\cite[Theorem 6.5]{exterior}), we obtain that
  \begin{equation}\label{panyushev_idnetity}
    \res_{\mathfrak{g}_0}V(\rho) \simeq
    V_0(\rho_0)\otimes V_0(p(\rho)-\rho_0)
  \end{equation}
  as $\mathfrak{g}_0$ modules. The results of this article are
  motivated by the above identity. Indeed if the saturation factor $d$
  of the pair $(G_0, G)$ is equal to one, our results
  describe the decomposition of the tensor product $V_0(\rho_0)\otimes
  V_0(p(\rho)-\rho_0)$ of $\mathfrak{g}_0$-representations. 
  
\section{A decomposition of \texorpdfstring{$V(\rho)$}{}}
Let $B$ and $T$ be the Borel subgroup and the maximal torus of $G$
such that the Lie algebras of $B$ and $T$ are $\mathfrak{b}$ and
$\mathfrak{h}$ respectively. Let $Z$ be the center of $G$. Let
$\Phi^\vee$ be the set of co-roots of $\mathfrak{g}$ with respect to
$\mathfrak{h}$. Let $\Phi^+$ (resp. $(\Phi^\vee)^+$) be the set of
positive roots (resp. co-roots) of $\Phi$ (resp. $\Phi^\vee$) for the
choice of the Borel subalgebra $\mathfrak{b}$. Let $\Delta$ and
$\Delta^\vee$ be a system of simple roots and simple co-roots in
$\Phi$ and $\Phi^\vee$ respectively. Let $G_0$ be the connected
subgroup of $G$ with Lie algebra $\mathfrak{g}_0$.  We denote by $B_0$
and $T_0$ the connected components containing identity of the groups
$B\cap G_0$ and $T\cap G_0$ respectively. Let $\Phi_0^+$ be the set of
positive roots of $\Phi_0$ with respect to the Borel subalgebra
$\mathfrak{b}_0$, and let $\Delta_0$ be the set of simple roots in
$\Phi^+_0$. Let ${\rm X}^\ast(T)$ (resp.  ${\rm X}^\ast(T_0)$) be the
character lattice of $T$ (resp. $T_0$).

Let $P\subset \mathfrak{h}^\ast$ and $P_0\subset \mathfrak{h}_0^\ast$
be the weight lattices of $\mathfrak{g}$ and $\mathfrak{g}_0$
respectively.  Let $P^+$ be the set of dominant weights of
$\mathfrak{g}$ with respect to the choices of $\mathfrak{h}$ and
$\mathfrak{b}$. For $\lambda \in P^+$, let $V(\lambda)$ be the
irreducible representation of $\mathfrak{g}$ with highest weight
$\lambda$. For $\mu\in P$, we denote by $V(\lambda)_\mu$ the
$\mu$-weight subspace of $V(\lambda)$.  The formal character
associated to the representation $V(\lambda)$, defined by the formal
sum $\sum_{\mu\in P}\dim(V(\lambda)_\mu)e^{\mu}$, is denoted by
$ch_\lambda$.  Similarly, let $P_0^+$ be the set of dominant weights
of $\mathfrak{g}_0$ with respect to the choice of $\mathfrak{h}_0$ and
$\mathfrak{b}_0$. Recall that we have
$p:\mathfrak{h}^\ast\rightarrow \mathfrak{h}_0^\ast$, the restriction
map.

Since $\mathfrak{g}_0$ is a simple Lie algebra but not simply laced,
$\Phi_0$ has two root lengths.  Let $\Phi_{0, s}$ (resp.
$\Phi_{0, l}$) be the set of short (resp. long) roots in
$\Phi_0$. Note that $\Phi_{0,s}$ and $\Phi_{0, l}$ are root subsystems
of $\Phi_0$, and the latter subsystem is closed.  We set
$\Phi_{0, s}^+$ to be the set $\Phi_0^+ \cap \Phi_{0, s}$. Let
$P_{0, s}^+$ be the set of dominant weights of the sub-root system
$\Phi_{0, s}$ with respect to the choice of positive roots
$\Phi_{0, s}^+$. The set $P_{0, s}^+$ is contained in $P_0^+$. We
write $\mu\preceq_s \lambda$, if $\lambda-\mu$ is a non-negative
integral linear combination of positive short roots. Let
$\rho_s=\frac{1}{2}\sum_{\beta \in \Phi_{0,s}^+}\beta$ and let
$\rho_l=\frac{1}{2}\sum_{\beta \in \Phi_{0,l}^+}\beta$. Let
$\Delta_{0, s}$ be the set $\Delta_0 \cap \Phi_{0, s}$ and let
$\Delta_{0,l}$ be the set of simple roots
$\Delta_0\backslash \Delta_{0,s}$.  We see that
$s_\alpha(\rho_s)=\rho_s-\alpha$, for $\alpha\in \Delta_{0, s}$ and
$s_\alpha(\rho_s)=\rho_s$ for $\alpha\in \Delta_{0, l}$, where
$s_\alpha$ is the reflection corresponding to the root $\alpha$. We
then have $\rho_{s}=\sum_{\alpha\in \Delta_{0, s}}\varpi_\alpha$ and
$\rho_l=\rho_0-\rho_s$. We note that $p(\rho)-\rho_0=\rho_s$ if
$\theta$ is of order $2$ and $p(\rho)-\rho_0=2\rho_s$ if $\theta$ has
order $3$. Let $(\mathfrak{g}_s, \mathfrak{g}_s\cap \mathfrak{h})$ be
the simple Lie algebra with its Cartan sub-algebra corresponding to
the root system $\Phi_{0, s}$.

We come to the crucial definition of saturation factor
associated to the pair $(G_0,G)$. 
\begin{definition}\label{sat_fac}
  A positive integer $d$ is called a {\it saturation factor} for the
  pair $(G_0, G)$ if for any $(\mu, \lambda)\in P_0^+\times P^+$ such
  that $\mu(z)\lambda(z)=1$, for all $z\in Z\cap G_0$, and
  $\ho_{G_0}(V_0(N\mu), V(N\lambda))\neq 0$ for some positive integer
  $N$, then $\ho_{G_0}(V_0(d\mu), V(d\lambda))\neq 0$.
\end{definition}
Let $C(G_0,G)$ be the set of all pairs
$(\mu, \lambda)\in P_0^+\times P^+$ such that
$\ho_{G_0}(V_0(N\mu), V(N\lambda))\neq 0$ for some positive integer
$N$. The set $C(G_0, G)$ is called an eigencone of the pair
$(G_0, G)$. Note that the finite dimensional representations of $G_0$
are self dual. Hence, we have
$$\dim_\mathbb{C}\ho_{G_0}(V_0(N\mu), V(N\lambda))
=\dim_\mathbb{C}(V_0(N\mu)\otimes V(N\lambda))^{G_0}.$$
Our result (Theorem \ref{thm_main}) generalises the main result of
\cite{components} on the components of the form $V(d\lambda)$ in the
tensor product $V(d\rho)\otimes V(d\rho)$, where $d$ is a saturation
factor of the group $G$ and $\lambda \in P^+$. We use Ressayre's work
(see \cite{ressayre}) on saturation problem in the context of
branching to prove the following theorem. 
\begin{theorem}\label{thm_main}
  Let $d$ be a saturation factor of the pair $(G_0, G)$, and let
  $\mu\in P_0^+$. The representation $V_0(d\mu)$ occurs in
  $\res_{\mathfrak{g}_0}V(d\rho)$ if and only if $\mu=\rho_0+\beta$
  where $\beta$ occurs as a weight of the
  $\mathfrak{g}_0$-representation $V_0(p(\rho)-\rho_0)$.
\end{theorem}
Before we prove the theorem, we recall some notations.  For any
$G_0$-dominant one parameter subgroup $\delta$ in $T_0$, let
$P_0(\delta)$ (resp. $P(\delta)$) be the parabolic subgroup associated
to $\delta$ inside $G_0$ (resp. $G$). Let $W_{0,P_0(\delta)}$
(resp. $W_{P(\delta)}$) be the Weyl group of $P_0(\delta)$
(resp. $P(\delta)$) and let $W_0^{P_0(\delta)}$
(resp. $W^{P(\delta)}$) be the set of minimal length representatives
in the cosets of $W_0/W_{0,P_0(\delta)}$
(resp. $W/W_{P(\delta)}$). Let
$i: G_0/P_0(\delta) \rightarrow G/P(\delta)$ be the natural
$G_0$-equivariant embedding and let
$i^*: H^*(G/P(\delta), \mathbb Z) \rightarrow H^*(G_0/P_0(\delta),
\mathbb Z)$ be the corresponding morphism of cohomology rings. For
$w' \in W_0^{P_0(\delta)}$ (resp.  $w \in W^{P(\delta)}$), let
$[X_{w'}] \in H^*(G_0/P_0(\delta), \mathbb Z)$ (resp.
$[X_w] \in H^*(G_0/P_0(\delta), \mathbb Z)$) be the cohomology class
of the sub-variety
$X_{w'}:=\overline{ B_0w'P_0(\delta)} \subseteq G_0/P_0(\delta)$
(resp. $X_{w}:=\overline{BwP(\delta)} \subseteq G/P(\delta)$).  A one
parameter subgroup in $T_0$ is said to be admissible if the hyperplane
of ${\rm X}_{\mathbb R}^*(T_0):={\rm X}^*(T_0) \otimes \mathbb R$
which it defines is spanned by the weights of the representation of
$T_0$ in $\mathfrak{g}/ \mathfrak{g}_0$ (see \cite[D\'efinition
4.7]{brion_res_bour}). For any complex algebraic group $H$, and for
any one parameter subgroup $\gamma:{\rm G}_m\rightarrow H$, we denote
by $\dot{\gamma}\in {\rm Lie}(H)$, the element $d\gamma(1)$.
\begin{proof}[Proof of Theorem \ref{thm_main}]
  We will use the results proved by Ressayre (see \cite [Theorem
  A]{ressayre}), to show that for any dominant weight $\mu$ such that
  $\mu=\rho_0+\beta$ with $V_0(p(\rho)-\rho_0)_\beta\neq 0$, the
  representation $V_0(d\mu)$ occurs in the representation
  $\res_{G_0}V(d\rho)$, where $d$ is any saturation factor for the
  pair $(G_0, G)$. We will use the exposition in Brion's Bourbaki
  article \cite{brion_res_bour}.

  The pair $(\lambda, \rho)\in P_0^+\times P^+$ belongs to $C(G_0, G)$
  if and only if, for any admissible and dominant one parameter
  subgroup $\delta$ of $T_0$ and for any
  $(w',w) \in W_0^{P_0(\delta)} \times W^{P(\delta)}$ such that
  \begin{enumerate}
  \item
    $i^*([X_w]).[X_{w'}]=[X_e] \in H^*(G_0/P_0(\delta), \mathbb Z)$
    and
  \item
    $(p(w^{-1}\rho)+p(\rho))(\dot{\delta})=(\rho_0-w'^{-1}\rho_0)(\dot{\delta})$,
    \end{enumerate}
 the following inequalities are satisfied:
  $$I_{\delta, w', w}:(p(w^{-1}\rho)+w'^{-1}\lambda)(\dot{\delta})
  \leq 0.$$ For the above result we refer to \cite[Th\'eor\`eme
  4.9]{brion_res_bour}.  We will show that condition $(2)$ already
  implies the inequalities $I_{\delta, w', w}$, and we will not
  explicitly use the cohomological condition $(1)$. Note that
  condition $(1)$ implies that
  $$(p(w^{-1}\rho)+p(\rho))(\dot{\delta})\leq
 ( \rho_0-w'^{-1}\rho_0)(\dot{\delta}).$$ (see \cite[Proposition
  2.3]{ressayre_branch} and \cite[equations
  (30)-(31)]{{kumar_survey}}). For the above results one may as well
  refer to the survey article: \cite[Theorem 33]{kumar_survey}.

  Let $(w',w) \in W_0^{P_0(\delta)} \times W^{P(\delta)}$ and $\delta$
  be as above.
  Since $(\rho_0-w'^{-1}\rho_0)(\dot{\delta})$ is equal to
  $(p(w^{-1}\rho)+p(\rho))(\dot{\delta})$, we have
  $$(p(w^{-1}\rho)+w'^{-1}\lambda)(\dot{\delta}) =
  (\rho_0-w'^{-1}\rho_0-p(\rho)+w'^{-1}\lambda)(\dot{\delta}).$$ 

  We apply these results to the present case, i.e., $\lambda=\mu$,
  where $\mu=\rho_0+\beta$ and $V_0(p(\rho)-\rho_0)_\beta\neq 0$. Note
  that $\mu\in {\rm X}^\ast(T_0)\subseteq P_0$. We have,
  $w'^{-1}\mu=w'^{-1}\rho_0+w'^{-1}\beta$. Since
  $V_0(p(\rho)-\rho_0)_{w'^{-1}\beta}\neq 0$ we have
  $w'^{-1}\beta \preceq p(\rho)-\rho_0$.  Thus,
  \begin{align*}
    (p(w^{-1}\rho)+w'^{-1}\mu)(\dot{\delta})=&
  (\rho_0-w'^{-1}\rho_0-p(\rho)+w'^{-1}\rho_0+w'^{-1}\beta)(\dot{\delta})\\
  &\leq
   (\rho_0-w'^{-1}\rho_0-p(\rho)+w'^{-1}\rho_0+p(\rho)-\rho_0)(\dot{\delta})=0.
  \end{align*}
  Hence, $(\mu, \rho)$ belongs to the eigencone $C(G_0, G)$. From the
  assumption that $\mu=\rho_0+\beta$, where
  $V_0(p(\rho)-\rho_0)_\beta\neq 0$, we get that $p(\rho)-\mu$ is
  contained in the root lattice of $G_0$. Thus we get that
  $\mu(z)\rho(z)=1$, for all $z\in Z\cap G_0$. We conclude that
  $V_0(d\mu)$ occurs as a sub-representation of
  $\res_{\mathfrak{g}_0}V(d\rho)$, for any saturation factor $d$ of
  the pair $(G_0, G)$.

  Conversely, assume that $V_0(d\mu)$ is contained in the representation
  $\res_{\mathfrak{g}_0}V(d\rho)$. Note that
  $\Phi_0^+=\Phi_{0,s}^+ \cup \Phi_{0,l}^+$. The restriction of the map
  $p:\mathfrak{h}^\ast\rightarrow (\mathfrak{h}_0)^\ast$, to the set $\Phi^+$
  surjects onto $\Phi_0^+$ and the preimage of a short root has cardinality
  equal to the order of $\theta$ and the preimage of a long root has
  cardinality $1$. So, from the Weyl Character formula, we have
  $$p(ch_{d\rho})=\prod_{\beta\in \Phi_0^+}
  (e^{d\beta/2}+\cdots+e^{-d\beta/2})^{m(\beta)},$$ where $m(\beta)$
  is the order of $\theta$ (which is either $2$ or $3$) if $\beta$ is
  short, and $m(\beta)=1$ if $\beta$ is long. Hence, we get that
$$p(ch_{d\rho})=\prod_{\beta\in \Phi_{0, s}^+}
(e^{d\beta/2}+\cdots+e^{-d\beta/2})^{m(\beta)-1}ch_{d\rho_0}.$$ Since
$p(ch_{d\rho})/ch_{d\rho_0}$ is invariant under the Weyl group, $W_0$,
of $G_0$, we get that
$$\prod_{\beta\in \Phi_{0, s}^+}
(e^{d\beta/2}+\cdots+e^{-d\beta/2})^{m(\beta)-1}=\sum_{i}n_ich_{\mu_i},$$
for some $n_i\in \mathbb{Z}$, and $\mu_i\preceq dp(\rho)-d\rho_0$
because $\rho_s=p(\rho)-\rho_0$ when $\theta$ has order $2$ and
$2\rho_s=p(\rho)-\rho_0$ when $\theta$ has order $3$. Therefore, we
get that $V_0(d\mu)$ is contained in $V_0(\mu_i)\otimes V(d\rho_0)$, where
$\mu_i\preceq dp(\rho)-d\rho_0$. Hence, the weight $d\mu$ is equal to
$d\rho_0+\beta$ where $V_0(\mu_i)_\beta\neq 0$. Now $d\mu-d\rho_0$
belongs to the convex hull of $W_0(\mu_i)$. Then, $d\mu-d\rho_0$
belongs to the convex hull of $W_0(dp(\rho)-d\rho_0)$. Thus we get
that $\mu-\rho_0$ belongs to the convex hull of
$W_0(p(\rho)-\rho_0)$. Since $\mu-\rho_0$ is a weight, and it belongs
to $W_0(p(\rho)-\rho_0)$, we get that $\mu-\rho_0$ is comparable to
$p(\rho)-\rho_0$ (see \cite[Proposition 8.44]{hall_lie}). Thus,
$V_0(p(\rho)-\rho_0)_{\mu-\rho_0}\neq 0$. Here, $V_0(p(\rho)-\rho_0)$ is
considered as a $\mathfrak{g}_0$-representation. This completes the
proof of the theorem.
\end{proof}
\begin{remark}\normalfont
  The above theorem holds true for any automorphism $\theta$ of order
  $2$ with some modification in the set up.  We need to choose a
  $\theta$-stable Cartan subalgebra $\mathfrak{h}$ and a
  $\theta$-stable Borel sub-algebra $\mathfrak{b}$ of $\mathfrak{g}$
  and the rest of the notations are as defined above. The proof is
  also the same. However, in what follows, we focus on the case of
  diagram automorphisms--in which case we describe the weights $\beta$
  occurring in the $\mathfrak{g}_0$-representation $V_0(p(\rho)-\rho_0)$
  such that $\rho_0+\beta$ is dominant.
   \end{remark}
\begin{remark}\normalfont
  Let us consider the case where $\mathfrak{g}$ is a simple Lie
  algebra of type ${\rm D}_4$, and let $\mathfrak{g}_0$ be a
  sub-algebra of type ${\rm G}_2$, fixed by a triality automorphism.
  Using the methods as in the above theorem, we can see that the
  irreducible sub-representations of $\res_{\mathfrak{g}_0}V(\rho)$
  are exactly the sub-representations of
  $$V_0(\rho_0)\otimes
  V_0(p(\rho)-\rho_0).$$
  In this small rank case, using Klimyk's
  formula, we check that the irreducible sub-representations which
  occur in the above tensor product are precisely the following set:
  $$\{V_0(\rho_0+\beta):\rho_0+\beta\in P_0^+
  \,\, \text{and} \,\, V_0(p(\rho)-\rho_0)_\beta\neq 0\}.$$
  \end{remark}

  In the following Lemma and Proposition, we assume that $\theta$ is
  of order $2$. The following lemma gives a characterisation of the
  weights occurring in the $\mathfrak{g}_0$-representation
  $V_0(p(\rho)-\rho_0)$. For any subset $\Psi\subseteq \Phi_0$, we
  denote by ${\rm sum}(\Psi)$ the vector $\sum_{v\in \Psi}v$.
\begin{lemma}\label{lemma}
  The space $V_0(p(\rho)-\rho_0)_\beta\neq 0$ if and only if
  $\beta=p(\rho)-\rho_0-{\rm sum}(\Psi)$, where $\Psi$ is a subset of
  $\Phi^+_{s, 0}$.
\end{lemma}
\begin{proof}
  Note that the character of the $\mathfrak{g}_0$-representation
  $V_0(p(\rho)-\rho_0)$ is given by
  $$\prod_{\beta\in \Phi_{s, 0}^+}(e^{\beta/2}+e^{-\beta/2}).$$
  The above is the formal character of the irreducible
  $\mathfrak{g}_{s}$-representation $V(\rho_s)$ with highest weight
  $\rho_s$. The lemma now follows from the result \cite[Lemma
  5.9]{lie_alg_cohomology} of Kostant.
\end{proof}

\begin{proposition}\label{proposition}
  Let $\mu\in P_0^+$. Then, $V_0(p(\rho)-\rho_0)_{\mu-\rho_0}\neq 0$ if
  and only if $\mu\preceq_s p(\rho)$ and
  $\mu-(2\rho_0-p(\rho))\in P_{0, s}^+$.
\end{proposition}
\begin{proof} Assume that $\mu \preceq_s p(\rho)$. Then
  $\mu-(2\rho_0-p(\rho)) \preceq_s p(\rho)-(2\rho_0-p(\rho))=2\rho_s$.
  Since $\mu-(2\rho_0-p(\rho)) \in P_{0, s}^+$, using the result
  \cite[Proposition 9]{components} we get
  that $$\mu-(2\rho_0-p(\rho))=\rho_s+\beta$$ where
  $V(\rho_s)_{\beta}\neq 0$.  Here, $V(\rho_s)$ is the
  $\mathfrak{g}_s$-representation with highest weight $\rho_s$.  Since
  the weights occurring in the $\mathfrak{g}_s$-representation
  $V(\rho_s)$ are the same as the weights occurring in the
  $\mathfrak{g}_0$-representation $V_0(p(\rho)-\rho_0)$, we have
  $\mu=\rho_0+\beta$ where $V_0(p(\rho)-\rho_0)_{\beta}\neq 0$.
  
  Conversely, assume that $\mu=\rho_0+\beta$ where
  $V_0(p(\rho)-\rho_0)_{\beta}\neq 0$. By Lemma \ref{lemma}, we have
  $\beta= p(\rho)-\rho_0-{\rm sum}(\Psi)$, where $\Psi$ is a subset of
  $\Phi^+_{s, 0}$. So
  $p(\rho)-\mu=p(\rho)-\rho_0-\beta
  =\rho_s-\beta=\rho_s-p(\rho)+\rho_0+{\rm sum}(\Psi)={\rm
    sum}(\Psi)$. So, $p(\rho)-\mu$ is a sum of positive short
  roots. Hence $\mu \preceq_s p(\rho)$.
  
  Write $\rho_0=\rho_s+\rho_l$.  Then $\mu=\rho_s+\rho_l+\beta$ and
  $\mu-(2\rho_0-p(\rho))=\rho_s+\beta$. For $\alpha\in \Delta_{0, s}$,
  we have $\langle \rho_l, \alpha \rangle=0$, and hence,
  $$\langle \rho_s+\beta, \alpha \rangle=\langle \mu-\rho_l, \alpha
  \rangle=\langle \mu , \alpha \rangle\geq 0.$$ This completes the
  proof of the proposition.
\end{proof}
\begin{remark}\normalfont
  Although we assume that $\mathfrak{g}$ is simple, we may consider
  the case where $\mathfrak{g}=\mathfrak{g}_0\times \mathfrak{g}_0$
  and $\theta$ to be the permutation automorphism of the two simple
  factors. With the rest of the notations being similar we get that
  $p(\rho)=2\rho_0$. This situation is considered in the paper \cite
  [Proposition 9]{components}. The hypothesis of Proposition
  \ref{proposition} is similar to the hypothesis in \cite [Proposition
  9]{components}.  However, when $\theta$ is an outer involution of a
  simple Lie algebra $\mathfrak{g}$, i.e., for the pair
  $(\mathfrak{g}, \mathfrak{g}_0)$, where $\mathfrak{g}_0$ is the
  fixed point sub-algebra of an outer automorphism of a simple Lie
  algebra $\mathfrak{g}$, there exist dominant weights $\mu$ such that
  $p(\rho)-\mu$ is a non-negative integral linear combination of short
  positive roots in $\Phi_0$ but $V_0(p(\rho)-\rho_0)_{\mu-\rho_0}=0$.

  Note that $p(\rho)$ is in the root lattice except in the case where
  $\mathfrak{g}$ is of the type ${\rm A}_{2n-1}$ and $\mathfrak{g}_0$
  is of the type ${\rm C}_n$, for $n$ is odd. When $p(\rho)$ is in the
  root lattice, we take $\mu=0$. In the case where the pair
  $(\mathfrak{g}, \mathfrak{g}_0)$ is of type $(A_{2n-1}, C_n)$, for
  $n$ is odd, we take $\mu=(2n-1)\varpi_1$; here, $\varpi_1$ is the
  first fundamental weight of $\mathfrak{g}_0$. Here the numbering of
  fundamental weights is according to the conventions in \cite[PLANCHE
  III]{bourbaki_lie_4_6}. Then it is easy to see that $p(\rho)-\mu$ is
  a non-negative integral linear combination of short positive
  roots. But, $V_0(p(\rho)-\rho_0)_{\mu-\rho_0}=0$ since $p(\rho)-\mu$
  is not of the form $p(\rho)-\rho_0-{\rm sum}(\Psi)$, for any
  $\Psi\subseteq \Phi_{s, 0}^+$.

 \end{remark}
\begin{remark}\normalfont
  Let $\mathfrak{g}=\mathfrak{g}_0\oplus \mathfrak{g_1}$ be the
  eigen-decomposition of $\theta$. From Kostant's description of
  $\wedge^\bullet \mathfrak{g}$, and Panyushev's work on the description of
  $\wedge^\bullet \mathfrak{g}_1$, we get that
  $$\res_{\mathfrak{g}_0}V(\rho)
  \simeq V_0(\rho_0)\otimes V_0(p(\rho)-\rho_0)$$ as $\mathfrak{g}_0$
  modules.  If $V_0(\mu)$ occurs in $\res_{\mathfrak{g}_0}V(\rho)$, then
  $\mu=\rho_0+\beta$, where $V_0(p(\rho)-\rho_0)_\beta\neq 0$. If $d=1$,
  for instance, in the cases where $(\mathfrak{g}, \mathfrak{g}_0)$
  are of the type: $({\rm D}_{n}, {\rm B}_{n-1})$,
  $({\rm E}_{6}, {\rm F}_{4})$ and $({\rm A}_{2n-1}, {\rm C}_{n})$,
  for $2\leq n\leq 5$, (see \cite{ressayre_exper}) then the main
  theorem gives the converse of this statement. In general, using
  computations on SAGE, we observed that the representation $V_0(\mu)$
  occurs in $\res_{\mathfrak{g}_0}V(\rho)$ if and only if
  $V_0(p(\rho)-\rho_0)_{\mu-\rho_0}\neq 0$.
\end{remark}
{\bf Acknowledgements:} We would like to thank Prof. Michel Brion for
the careful reading of the manuscript and for some helpful
comments. We also thank Prof. Shrawan Kumar for some useful comments.

\bibliography{../biblio} \bibliographystyle{abbrv}
Department of Mathematics and Statistics,\\
    Indian Institute of Technology Kanpur,\\
    U.P. India, 208016.\\
    email:\ \texttt{nsantosh@iitk.ac.in}; \texttt{santosha@iitk.ac.in}
\end{document}